\documentclass{article}

\usepackage[total={7in, 9in}]{geometry}
\usepackage{amsmath,amsthm,amssymb}
\usepackage{algorithm}
\usepackage{algpseudocode}
\usepackage{authblk}
\usepackage{lineno}
\usepackage{booktabs}
\usepackage{siunitx}
\usepackage{multirow}%
\usepackage{pgfplots}
\usepackage{rotating}
\usepackage[numbers,sort&compress]{natbib}
\pgfplotsset{compat=1.18}
\usepackage{hyperref}

\newcommand*{\dif}{\mathop{}\!\mathrm{d}}
\newtheorem{theorem}{Theorem}[section] 
\newtheorem{definition}{Definition}[section] 
\newtheorem{lemma}{Lemma}[section]

\hypersetup{
	colorlinks=true,
	linkcolor=blue,
	citecolor=blue,
	urlcolor=orange
}

\title{A $p$-adaptive treecode solution of the Poisson equation in the general domain} 
\author[1]{Zixuan Cui}
\author[1,2,\thanks{Corresponding author: \texttt{leiyang@must.edu.mo}}]{Lei Yang}

\affil[1]{School of Computer Science and Engineering,
	Macau University of Science and Technology, Macao SAR 999078, China.}
\affil[2]{Macau University of Science and Technology Zhuhai MUST Science and Technology Research Institute, Zhuhai 519031, China.}
\begin{document}
%	\linenumbers
	\maketitle
	\begin{abstract}
		
		Raising the order of the multipole expansion is a feasible approach for improving the accuracy of the treecode algorithm. However, a uniform order for the expansion would result in the inefficiency of the implementation, especially when the kernel function is singular. In this paper, a $p$-adaptive treecode algorithm is designed to resolve the efficiency issue for problems defined on a general domain. Such a $p$-adaptive implementation is realized through i). conducting a systematical error analysis for the treecode algorithm, ii). designing a strategy for a non-uniform distribution of the order of multipole expansion towards a given error tolerance, and iii). employing a hierarchy geometry tree structure for coding the algorithm. The proposed $p$-adaptive treecode algorithm is validated by a number of numerical experiments, from which the desired performance is observed successfully, i.e., the computational complexity is reduced dramatically compared with the uniform order case, making our algorithm a competitive one for bottleneck problems such as the demagnetizing field calculation in computational micromagnetics.

	\end{abstract}

	\section{Introduction}
	The solution of Poisson equation is a critical task in many fields of science and engineering from quantum physics and chemistry to materials science and biology \cite{kuang2024towards,motamarri2013higher,wilson2022tabi}. In these fields, problems such as the computation of the demagnetizing field in micromagnetism and the computation of Hartree potential in quantum physics often require a fast and accurate solution of Poisson equation. However, the use of direct numerical integration methods to compute the fundamental solution of the Poisson equation with long-range interactions leads to a computational complexity of $O(N^2)$, where $N$ is the number of mesh elements in computational domain $\Omega$. When $N$ is large, this complexity is unacceptable. 
	
	To address this issue, a series of fast algorithms have been proposed as potential solutions to the Poisson equation, including Multigrid method \cite{teunissen2023geometric}, Finite Element Method (FEM) \cite{larson2013finite}, Fast Fourier Transform (FFT) \cite{wu2013optimized}, nonuniform fast Fourier transform (NUFFT) \cite{greengard2004accelerating}, treecode algorithm \cite{christlieb2004efficient, cui2024treecode}, and the Fast Multipole Method (FMM) \cite{fryklund2023fmm}, etc. These algorithms reduce the complexity from $O(N^2)$ to $O(N\log N)$ even $O(N)$, respectively. While both the multigrid method and FEM are very efficient in solving the Poisson equation, approximation of boundary values with an integral form by direct summation is an $O(N^{4/3})$ operation in three dimensions, such as the calculation of the demagnetizing field in the Landau-Lifshitz-Gilbert (LLG) equation \cite{AAMM-11-1048, yang2021framework}. In addition, FFT is an efficient solver for Poisson equation based on the potential theory with complexity of $O(N\log N)$ but is limited to regular computational domains, such as rectangular grids and periodic boundary conditions. In practical applications, computational domains are often complex and irregular, such as when simulating dynamic magnetization in ferromagnetic materials with defects \cite{chen2005effect}. NUFFT is easy to deal with irregular domain, but the accuracy of this approach is hampered due to the singularity in the interaction kernel, such as $3$D Coulomb interaction. Therefore, a high-order discretization by generalized Gaussian quadrature \cite{ma1996generalized} or coordinate transform \cite{jiang2014fast, bao2015computing} has to be used for canceling the singularity, which unavoidably adds some computational costs. Treecode is a competitive one for problems defined in complex domain with $O(N\log N)$, and FMM is even better due to its optimal computational complexity $O(N)$. However, in parallel environment, due to the translation method used in FMM, the two method, treecode and FMM, are comparable \cite{yokota2011comparing}.

	 There is a challenge for solving Poisson equation based on the potential theory, where singularity is encountered in modeling a variety of problems. One of the typical cases is the Coulomb interaction arising in Schr\"{o}dinger equation and Kohn--Sham equation \cite{bardos2000weak, kohn1965self, sheehy2007quantum}. Dipole-dipole interactions with same or different dipole orientation arises in quantum chemistry, dipolar Bose-Einstein condensation, etc. \cite{bao2012gross,haken1995molecular, lahaye2009physics}.  As a result, an accurate quadrature has to be used in order to handle the singularity of kernel function. However, a direct use of the standard FFT or NUFFT, a phenomenon know as “numerical locking” occurs, limiting the achievable precision \cite{tikhonenkov2008anisotropic}. The treecode originally approximate the interaction with low expansion for efficiency. 
	 In addition to the accuracy issue, the approximation of the treecode algorithm lacks global continuity, which leads to energy drift when the order in multipole expansion is low, such as in molecular dynamics simulations \cite{skeel2002multiple, sagui2001multigrid, tlupova2022effect}. 
		
	Uniformly increasing the expansion order improves the accuracy and stability of the treecode algorithm but incurs a computational cost of \( O(p^3) \), where \( p \) is the expansion order. Enlarging the near-field region for direct summation can also enhance accuracy but at a significant computational expense. To address these issues, we focus on a non-uniform adjustment of the expansion order in multipole expansions to achieve an efficient and accurate solution to the Poisson equation.

	In this paper, we propose a $p$-adaptive treecode algorithm designed to improve the efficiency of solving the Poisson equation in general domains, by resolving following issues. First of all, a systematic error analysis is conducted, providing deeper insights into local errors within the hierarchical tetrahedral mesh. Furthermore, we develop a strategy for the non-uniform distribution of the multipole expansion order based on a given error tolerance. Comparing with traditional treecode methods that employ a uniform expansion order, our approach adaptively adjusts the expansion order according to the relative distance between target and source points, thereby avoiding local precision redundancy or insufficiency that often arises from uniform-order approximations. This strategy enhances computational efficiency while maintaining accuracy, particularly in large-scale problems or cases requiring high precision.
	
	One of the key challenges in implementing the proposed algorithm is efficiently constructing the connection between the mesh and the tree structure. To address this, we code the algorithm by a hierarchy geometry tree data structure, which facilitates the efficient generation of the tree structure within the tetrahedral mesh and manages interactions between neighbor and non-neighbor elements in the treecode algorithm. This tetrahedral mesh-based implementation allows flexible application of the algorithm to general domains. Furthermore, a trick for additional acceleration through direct summation is proposed.
	
	The proposed $p$-adaptive treecode algorithm is validated through various numerical experiments, which demonstrate significant performance improvements. Specifically, the computational complexity is dramatically reduced compared to the uniform order case, making the algorithm highly competitive for bottleneck problems such as demagnetizing field calculations in computational micromagnetics.
	
	The organization of this paper is as follows. In Section \ref{section 2}, we introduce the fundamental solution of the Poisson equation and review the treecode algorithm with uniform order. In Section \ref{section 3}, we present the error analysis of the multipole expansion for the treecode algorithm and describe the details of the $p$-adaptive treecode algorithm. Section \ref{section 4} examines the accuracy and efficiency of the proposed algorithm, comparing it with the uniform order case and evaluating the impact of various parameters. Finally, conclusions are presented in Section \ref{section 5}.

	\section{Poisson equation and its treecode solver}
	\label{section 2}
	In this section, we first review the unbounded Poisson equation and its fundamental solution. Then the mesh-based treecode algorithm is described by using multipole expansion with uniform order.
	\subsection{Poisson equation}
	Consider a general unbounded Poisson equation with natural boundary condition defined on $\mathbb{R}^d$ ($d = 2, 3$):
	\begin{equation} \left\{
		\begin{array}{lr} -\Delta u(\mathbf{x}) = f(\mathbf{x}), \\
			\lim\limits_{\mid\mathbf{x}\mid\rightarrow\infty}u(\mathbf{x}) = 0.
		\end{array} \right.
		\label{Poisson eq}
	\end{equation} 
	The above equation appears in the calculation of demagnetizing, and Hartree potential, etc. Using the Green's function approach, the fundamental solution of Eq.~(\ref{Poisson eq}) can be given by the convolution written as
	\begin{equation}
		u(\mathbf{x}) = \int_{\mathbb{R}^d}
		G(\mathbf{x},\mathbf{y})f(\mathbf{y})\dif\mathbf{y},
		\label{fundamental solution} 
	\end{equation}
	with
	\begin{equation} 
		G(\mathbf{x},\mathbf{y}) = \left\{
		\begin{aligned} &\dfrac{1}{2\pi} \ln\dfrac{1}{|\mathbf{x}-\mathbf{y}|}, &d=2,\\
			&\dfrac{1}{4\pi|\mathbf{x}-\mathbf{y}|}, &d=3.
		\end{aligned} \right.
	\end{equation}
	Here, the Green's function $G(\mathbf{x},\mathbf{y})$ satisfies
	\begin{equation} 
		\left\{
		\begin{array}{lr} 
			\Delta G = -\delta(\mathbf{x}-\mathbf{y}),\\ 
			\lim\limits_{|\mathbf{x}|\rightarrow\infty}G = 0,
		\end{array} 
		\right.
	\end{equation}
	where $\delta(\mathbf{x}-\mathbf{y})$ is the Dirac delta function.
	
	To solve Eq.~(\ref{fundamental solution}), the problem has to be formulated on a bounded computational domain $\Omega \subset \mathbb{R}^d$. A domain truncation approach is employed to handle the integration over unbounded domain. If the solution $u(\mathbf{x})$
	exhibits rapid decay as $|\mathbf{x}|$ tends to infinity, the truncation error is defined as 
	\begin{equation}
		E_{DT} = \max_{\mathbf{y}\in\partial\Omega}|u(\mathbf{y})|,
	\end{equation}
	or otherwise
	\begin{equation}
		E_{DT}  = \max_{\mathbf{y}\in \mathbb{R}^3\setminus\Omega} |u(\mathbf{y})|.
	\end{equation}
	The radius of computational domain should be chosen to be as large as possible to minimize truncation errors. 
	\subsection{Treecode solver with a uniform order of multipole expansion}
	In the subsequent section, we focus on the $3$D problem. Specifically, we aim to solve the following convolution integral over the truncated domain $\Omega$,
	\begin{equation}
		u(\mathbf{x}) = \int_\Omega G(\mathbf{x},\mathbf{y})f(\mathbf{y})\dif\mathbf{y}.
		\label{truncated int}
	\end{equation}
	To discretize the domain $\Omega$, a tetrahedral mesh $\mathcal{T}_h = \{\tau_0, \cdots, \tau_{N-1}\}$ has been applied, where $\tau_i$ is the $(i-1)$-th element of the mesh. For the barycenter $\mathbf{x}_i$ of each element $\tau_i$, Eq.~(\ref{truncated int}) can be approximated by a direct summation method, denoted by
	\begin{equation}
		u(\mathbf{x}_i) \approx \sum_{K\in\mathcal{T}_h} \sum_{l=1}^{N_q}G(\mathbf{x}_i,\mathbf{y}_l)f(\mathbf{y}_n)\omega_l,
		\label{discrete form}
	\end{equation}
	where $K$ represents the tetrahedron in the mesh, $\mathbf{y}_l$ is the $l$-th quadrature point of the element $K$, $\omega_l$ is the associated quadrature weight, and $N_q$ is the total number of quadrature points. The direct calculation of $u(\mathbf{x})$ involves $O(N^2)$ operations due to non-local interactions, making it impractical for large $N$.

	To improve computational efficiency, the treecode algorithm can be employed, which reduces the complexity to $O(N\log N)$ by employing a divide-and-conquer strategy. The treecode algorithm consists of two steps: 1. building a hierarchical data structure and 2. approximating the far-field interactions by multipole expansion.
	
	In the first part, treecode algorithm divides the computational domain into hierarchical domains and then forming a tree structure. The treecode algorithm based on tetrahedral mesh constructs a multi-level tree structure by progressively refining the mesh. Suppose the initial mesh is given by $\mathcal{T}_0 = \{\tau_0, \tau_1\}$. After refinement, the second-level mesh becomes $\mathcal{T}_1 = \{\tau_{00}, \cdots, \tau_{07}, \tau_{10}, \cdots, \tau_{17}\}$. By repeating this process $l-1$ times, a complete octree structure with $l$ levels is obtained. In this tree, the barycenters of the leaf-level tetrahedrons are considered as target points for computation.
	\begin{figure}[!ht]
		\centering
		\includegraphics[width = 0.4\textwidth]{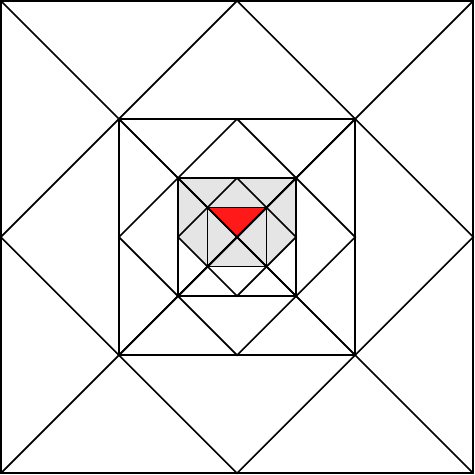}
		\caption{Treecode structure for a single target element with red background. Gray triangles are neighbor elements and white elements are non-neighbor elements.}
		\label{treecode structure}
	\end{figure} 
	For a given target point, treecode allows elements in far-field domain to be merged at higher levels, reducing the computational cost. Fig.~\ref{treecode structure} illustrates the mesh distribution for a single target element in 2D, where far-field elements are merged based on their distance from the target element. The elements in Fig.~\ref{treecode structure} are classified into two categories: neighbor elements and non-neighbor elements. Neighbor elements share common vertices with the target element, forming the near-field domain. All other elements are classified as non-neighbor elements, which form the far-field domain. This hierarchical mesh, constructed via a bottom-up recursive process, is the core data structure of treecode. One of the challenges in treecode is the efficient implementation and management of this tree structure. We address this by employing the hierarchy geometry tree method, with further details provided in Section \ref{section 3}.
	
	The second part of the treecode algorithm uses multipole expansion to approximate far-field interactions. In the near-field region ($\Omega_N$), interactions are computed directly using numerical quadrature due to the small number of elements. In the far-field region ($\Omega_F$), interactions are approximated using multipole expansions, where contributions come from non-neighbor elements. The resulting solution at a target point is obtained by summing the near-field contributions and the far-field approximations, as shown in Eq.~(\ref{NF_integral}).

	\begin{equation}
		u(\mathbf{x}_i) \approx \sum_{K\in\Omega_N}\sum_{l=1}^{N_q}G(\mathbf{x}_i,\mathbf{y}_l)f(\mathbf{y}_l)\omega_l + \sum_{K\in\Omega_F}\sum_{\|\mathbf{k}\|=0}^{p}a^{\mathbf{k}}(\mathbf{x}_i,\mathbf{y}_c)m_c^{\mathbf{k}},
		\label{NF_integral}
	\end{equation}
	where the $\mathbf{k}$-th order expansion coefficient $a^{\mathbf{k}}(\mathbf{x}_i,\mathbf{y}_c)$ of kernel function $G$ at the barycenter $\mathbf{y}_c$ of the element $K$, written as
	\begin{equation}
		a^{\mathbf{k}}(\mathbf{x}_i,\mathbf{y}_c) = \dfrac{1}{\mathbf{k}!} D^{\mathbf{k}}_{\mathbf{y}}G(\mathbf{x}_i,\mathbf{y}_c),
	\end{equation}
	where $D^{\mathbf{k}}_{\mathbf{y}}$ denotes $\mathbf{k}$-th order differential operator with respect to $\mathbf{y}$ and the $\mathbf{k}$th moment of the element $K$ denoted by
	\begin{equation}
		m_c^{\mathbf{k}} = \sum_{l=1}^{N_q}(\mathbf{y}_l - \mathbf{y}_c)^{\mathbf{k}}f(\mathbf{y}_l)\omega_l.
		\label{moment}
	\end{equation}
	In addition, Cartesian multi-index notation has been used with $\mathbf{k} = (k_1,k_2,k_3)$,
	$k_i \in \mathbb{N}$, $\|\mathbf{k}\| = k_1 + k_2 + k_3$,
	$\mathbf{k}! = k_1k_2k_3$, $\mathbf{y} = (y_1, y_2, y_3)$,
	$ y_i \in \mathbb R$, $\mathbf{y}^{\mathbf{k}} = y_1^{k_1}y_2^{k_2}y_3^{k_3}$,
	$D^{\mathbf{k}}_{\mathbf{y}} = D_{y_1}^{k_1}D_{y_2}^{k_2}D_{y_3}^{k_3}$.
	In implementing higher-order multipole expansions, we follow the recurrence relations in \cite{li2009cartesian} for the expansion coefficients $a^{\mathbf{k}}(\mathbf{x},\mathbf{y}_c)$,
	\begin{equation}
		\|\mathbf{k}\|\left|\mathbf{x} - \mathbf{y}_c\right|^2 a^{\mathbf{k}} - (2\|\mathbf{k}\| - 1)\sum_{i=1}^{3}(x_i - y_i)a^{\mathbf{k} - \mathbf{e}_i} +(\|\mathbf{k}\| - 1) \sum_{i=1}^{3}a^{\mathbf{k} - 2\mathbf{e}_i} = 0,
		\label{recurrence relation}
	\end{equation}
	where $\mathbf{e}_i$ are the Cartesian basis vectors.
	
	However, it is important to note that even with the precomputation of the moment $m_c^{\mathbf{k}}$, each $p$th order multipole expansion approximation still requires $O(p^3)$ operations. This computational cost becomes substantial when higher-order expansions are employed to enhance algorithm accuracy and for large-scale computations. In order to reduce this additional burden, we adopt non-uniform distribution of the order of multipole expansion, selecting the minimum order $p$ that satisfies a specified accuracy criterion. In the following section, we present an error analysis of the treecode algorithm, which serves as a guideline for the accuracy criterion.

	\section{Treecode with non-uniform order of multipole expansion}
	\label{section 3}
	In this section, we perform an error analysis of the treecode algorithm, providing an accuracy criterion for adaptively selecting the expansion order $p$. We also outline the implementation details and complexity of the $p$-adaptive treecode algorithm. Additionally, a trick for further acceleration using direct summation is proposed.

	\subsection{An error analysis for the treecode algorithm}
	In this subsection, we provide a systematic error analysis for the treecode algorithm, which will play an important role in the following subsections. Firstly, we give a definition of neighbor and non-neighbor elements according to the distance between elements.
	
	\begin{definition}
		Let $r_y$ denote the maximum radius of element $K$ as shown in Fig.~\ref{distance}, and $R_K$ be the distance between the barycenter $\mathbf{y}_c$ of $K$ and the target point $\mathbf{x}_i$. If the ratio 
		\begin{equation}
			r_K \equiv \dfrac{r_y}{R_K} \le 1,
		\end{equation}
		then $K$ is a non-neighbor element belonging to the far-field domain $\Omega_F$. Otherwise, $K$ is a neighbor element belonging to the near-field domain $\Omega_N$.
		\label{def 1}
	\end{definition}
	
	\begin{figure}[!ht]
		\centering
		\includegraphics[width=0.6\textwidth]{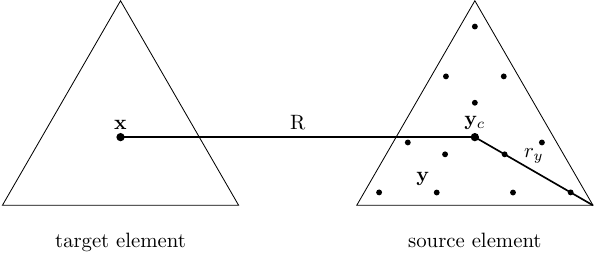}
		\caption{The well-separated target element and source element. Here, $\mathbf{x}$ and $\mathbf{y}_c$ denote the barycenters of the two elements, $R = \left|\mathbf{x}-\mathbf{y}_c\right|$ is the distance between the barycenters, $\mathbf{y}$ represents the quadrature points in the source element, and $r_y$ is the maximum radius of the source element.}
		\label{distance}
	\end{figure}
	
	To prove Theorem~\ref{Th 1}, we rely on the equivalence of the remainder terms in the multipole and Gegenbauer expansions, as formalized in the following lemma.
	\begin{lemma}
		Let $\mathbf{x}$ and $\mathbf{y}$ be two distinct points in \( \mathbb{R}^3 \), and consider the kernel function $ 1/|\mathbf{x} - \mathbf{y}|$. Suppose $|\mathbf{x}|<|\mathbf{y}|$, the multipole expansion of the kernel function in Cartesian coordinates around $ \mathbf{y} = \mathbf{y}_c$ and the Gegenbauer polynomial expansion have equivalent remainder terms. Specifically, for the expansions up to the order $p$, we have
		\begin{equation}
			\sum_{k=p+1}^{\infty}\sum_{\|\mathbf{k}\|=k}\dfrac{1}{\mathbf{k}!} D^{\mathbf{k}}_{\mathbf{y}} \dfrac{1}{|\mathbf{x} - \mathbf{y}_c|}(\mathbf{y}-\mathbf{y}_c)^{\mathbf{k}}= 	\sum_{k=p+1}^{\infty}\dfrac{|\mathbf{x}|^k}{|\mathbf{y}|^{k+1}}C_k^{1/2}\left(\dfrac{\mathbf{x}\cdot\mathbf{y}}{|\mathbf{x}|\cdot|\mathbf{y}|}\right),
			\label{equivalent}
		\end{equation}
		where $C_k^{1/2}(\mathbf{y})$ is the Gegenbauer polynomial with $k$th degree.
		\label{lemma 1}
	\end{lemma}

	\begin{proof}
		The kernel function $1/|\mathbf{x} - \mathbf{y}|$ can be expanded at the center of source element by multipole expansion in Cartesian coordinate, written as
		
		\begin{equation}
			\frac{1}{\left|\mathbf{x} - \mathbf{y}\right|} = \sum_{\|\mathbf{k}\|=0}^{\infty}T^{\mathbf{k}}(\mathbf{x},\mathbf{y}_c)(\mathbf{y}-\mathbf{y}_c)^{\mathbf{k}},
			\label{taylor expansion}
		\end{equation}
		with 
		\begin{equation}
			T^{\mathbf{k}}(\mathbf{x},\mathbf{y}_c) = \dfrac{1}{\mathbf{k}!} D^{\mathbf{k}}_{\mathbf{y}} \dfrac{1}{|\mathbf{x} - \mathbf{y}_c|},
		\end{equation}
		and the coefficient $T^{\mathbf{k}}(\mathbf{x},\mathbf{y}_c)$ satisfies the recurrence relations (\ref{recurrence relation}).
		Similarly, the Gegenbauer polynomial expansion of the kernel function can be given by \cite{duan2001adaptive}:
		\begin{equation}
			\frac{1}{\left|\mathbf{x} - \mathbf{y}\right|} = \sum^{\infty}_{k=0} \frac{|\mathbf{x}|^k}{|\mathbf{y}|^{k+1}}C_k^{1/2}\left(\frac{\mathbf{x}\cdot\mathbf{y}}{|\mathbf{x}|\cdot|\mathbf{y}|}\right),
			\label{Gegenbauer}
		\end{equation}
		which is convergent when $|\mathbf{x}|<|\mathbf{y}|$. It is worth noting that the expansion coefficients $C_k^{1/2}(y)$ exhibit a recurrence relation analogous to that observed in Eq.~(\ref{recurrence relation}), denoted by
		\begin{equation}
			kC_k^{1/2}(y) - (2k-1)\theta C_{k-1}^{1/2}(y)+(k-1)C_{k-2}^{1/2}(y) = 0
			\label{RC_Legendre}
		\end{equation}
		for $k\geq2$, where $C_0^{1/2} = 1$ and $C_1^{1/2} = y$. 
		% As the two expansion coefficients are obtained by expanding the kernel function in Cartesian and spherical coordinates, respectively, this result is to be expected.

		% Next, we give the error bound of $p$th-order multipole approximation of the kernel function. 
		
		Comparing Eq.~(\ref{taylor expansion}) with Eq.~(\ref{Gegenbauer}) and then combining them with the recurrence relations (\ref{recurrence relation}) and (\ref{RC_Legendre}), we obtain the following equality
		\begin{equation}
			\sum_{\|\mathbf{k}\|=k}T^{\mathbf{k}}(\mathbf{x},\mathbf{y}_c)(\mathbf{y}-\mathbf{y}_c)^{\mathbf{k}}= \left|\dfrac{|\mathbf{x}|^k}{|\mathbf{y}|^{k+1}}C_k^{1/2}\left(\dfrac{\mathbf{x}\cdot\mathbf{y}}{|\mathbf{x}|\cdot|\mathbf{y}|}\right)\right|.
		\end{equation}
		Since the $k$-th terms in both series are identical, it follows that their respective remainder terms must also be equivalent. Consequently, the remainder of the multipole expansion in Cartesian coordinates satisfies Eq.~(\ref{equivalent}).
	\end{proof}

	 By Lemma 1, the remainder terms of the multipole expansion and the Gegenbauer expansion are shown to be equivalent. Leveraging this equivalence, together with the orthogonality and recurrence properties of the Gegenbauer polynomials, we analyze the error introduced by the multipole expansion in the treecode algorithm.

	\begin{theorem}
		Let \( T_h \) be a tetrahedral mesh of the computational domain \( \Omega \) with characteristic size \( h \), and let \( \{\mathbf{x}_i\} \) denote the barycenters of the elements in \( T_h \). Let \( \{K\} \) be the set of source elements associated with \( x_i \), and assume the following:
		\begin{itemize}
			\item The treecode algorithm computes the near-field interactions exactly.
			\item The far-field contributions are approximated using a \( p \)-th order multipole expansion in Cartesian coordinates.
			\item The source function \( f(\mathbf{x}) \) is bounded above by a positive constant \( F > 0 \) in \( \Omega \).
		\end{itemize}
		
		Under these assumptions, the error \( E(\mathbf{x}_i) \) at the target point \( \mathbf{x}_i \), arising from the multipole approximation of far-field interactions, satisfies:
		\begin{equation}
			E(\mathbf{x}_i) \leq C \left| \sum_{K \in \Omega_F} \frac{r_K^{p+1} |\Omega_K|}{R_K (1 - r_K)} \right|,
		\end{equation}
		where \( C \) is a constant depending on \( f(\mathbf{x}) \), \( |\Omega_K| \) is the volume of the element \( K \), and \( r_K \), \( R_K \) are defined as in Definition~\ref{def 1}.
		\label{Th 1}
	\end{theorem}

	\begin{proof}
		
		According to Lemma \ref{lemma 1}, the remainders of the multipole expansion and the Gegenbauer polynomial expansion are equivalent under the condition that $|\mathbf{x}|<|\mathbf{y}|$. However, in the treecode, it cannot always be guaranteed that the source element is farther from the origin than the target element, as depicted in Fig.~\ref{distance}. Consequently, the convergence condition of Eq.~(\ref{Gegenbauer}) may not be satisfied, requiring a modification to the method. Consider the relative positions between the source and target points as the basis for the expansion terms, denoted as
		\begin{equation}
			\begin{aligned}
				\frac{1}{\left|\mathbf{x} - \mathbf{y}\right|} &=  \frac{1}{\left|\mathbf{y}-\mathbf{y}_c - (\mathbf{x}-\mathbf{y}_c)\right|}\\
				~\\
				&=
				\sum^{\infty}_{k=0} \dfrac{|\mathbf{y}-\mathbf{y}_c|^k}{|\mathbf{x}-\mathbf{y}_c|^{k+1}}C_k^{1/2}\left(\frac{\mathbf{y}-\mathbf{y}_c}{|\mathbf{y}-\mathbf{y}_c|}\cdot\frac{\mathbf{x}-\mathbf{y}_c}{|\mathbf{x}-\mathbf{y}_c|}\right),
			\end{aligned}
			\label{legendre expansion}
		\end{equation}
		which results in $\left|\mathbf{y}-\mathbf{y}_c\right| < \left|\mathbf{x}-\mathbf{y}_c\right|$, ensuring that the convergence condition is satisfied after modification. This condition is naturally met when the partitioned mesh is regularized and is commonly referred to as the multipole acceptance criterion (MAC), 
		namely
		
		\begin{equation}
			\dfrac{\left|\mathbf{y}-\mathbf{y}_c\right|}{\left|\mathbf{x}-\mathbf{y}_c\right|}\le  r_K.
			\label{MAC}
		\end{equation}
		Using the fact that $C_k^{1/2}(y) \le 1$ for $|y|\le1$, and 
		\begin{equation}
			\left|\dfrac{\mathbf{y}-\mathbf{y}_c}{|\mathbf{y}-\mathbf{y}_c|}\cdot\frac{\mathbf{x}-\mathbf{y}_c}{|\mathbf{x}-\mathbf{y}_c|}\right| < 1,
		\end{equation}
		we obtain the following inequality
		\begin{equation}
			\sum_{k=p+1}^{\infty}\dfrac{|\mathbf{x}|^k}{|\mathbf{y}|^{k+1}}C_k^{1/2}\left(\dfrac{\mathbf{x}\cdot\mathbf{y}}{|\mathbf{x}|\cdot|\mathbf{y}|}\right) \le \sum^{\infty}_{k=p+1} \frac{r^k_K}{R_K} =\frac{r_K^{p+1}}{R_K(1-r_K)}.
		\end{equation}
		Consequently, the following inequality can be derived from Lemma \ref{lemma 1}:
		\begin{equation}
			\sum_{k=p+1}^{\infty}\sum_{\|\mathbf{k}\|=k}\left|T^{\mathbf{k}}(\mathbf{x},\mathbf{y}_c)(\mathbf{y}-\mathbf{y}_c)^{\mathbf{k}}\right|
			\le \frac{r_K^{p+1}}{R_K(1-r_K)}.
		\end{equation}
		Given that the source function $f$ is bounded above in $\Omega$ and suppose the near-field integration is computed exactly, an error bound for the treecode algorithm can be derived as follows:
		\begin{equation}
			\begin{aligned}
				E(\mathbf{x}_i)=&\left|\int_{\Omega}G(\mathbf{x}_i,\mathbf{y})f(\mathbf{y})\dif \mathbf{y} - \left[\int_{\Omega_N}G(\mathbf{x}_i,\mathbf{y})f(\mathbf{y})\dif \mathbf{y} +  \int_{\Omega_F}\sum_{\|\mathbf{k}\|=0}^{p}a^{\mathbf{k}}(\mathbf{x}_i,\mathbf{y}_c)(\mathbf{y}-\mathbf{y}_c)^{\mathbf{k}}f(\mathbf{y})\dif \mathbf{y}\right]\right|\\
				~\\
				=&\left|\sum_{K\in\Omega_F}\left(\int_{\Omega_K}G(\mathbf{x}_i,\mathbf{y})f(\mathbf{y})\dif \mathbf{y} - \int_{\Omega_K}\sum_{\|\mathbf{k}\|=0}^{p}a^{\mathbf{k}}(\mathbf{x}_i,\mathbf{y}_c)(\mathbf{y}-\mathbf{y}_c)^{\mathbf{k}}f(\mathbf{y})\dif \mathbf{y}\right)\right|\\
				~\\
				=&\left|\sum_{K\in\Omega_F}\int_{\Omega_K}\sum_{\|\mathbf{k}\|=p+1}^{\infty}a^{\mathbf{k}}(\mathbf{x}_i,\mathbf{y}_c)(\mathbf{y}-\mathbf{y}_c)^{\mathbf{k}}f(\mathbf{y})\dif \mathbf{y}\right|\\
				~\\
				\le&\dfrac{1}{4\pi}\left|\sum_{K\in\Omega_F}\int_{\Omega_K}\dfrac{r_K^{p+1}}{R_K(1-r_K)}f(\mathbf{y})\dif \mathbf{y}\right|\\
				~\\
				\le&\dfrac{1}{4\pi}\left|\sum_{K\in\Omega_F}\dfrac{r_K^{p+1}|F|\left|\Omega_K\right|}{R_K(1-r_K)}  \right|\\
				~\\
				=&C\left|\sum_{K\in\Omega_F}\dfrac{r_K^{p+1}\left|\Omega_K\right|}{R_K(1-r_K)}  \right|,
			\end{aligned}
		\end{equation}
		where $C = |F|/4\pi$.
	\end{proof}

	Theorem~\ref{Th 1} establishes a rigorous framework for quantifying the error introduced by the multipole expansion in the treecode algorithm. This result enhances the theoretical understanding of the accuracy of treecode algorithm and serves as the foundation for the improvement of accuracy. Based on the error bound, the next section presents a $p$-adaptive strategy that balances computational efficiency and accuracy by dynamically adjusting the expansion order.
	
	\subsection{A $p$-adaptive implementation}
	\label{sub p-adaptive implementation}
	In this subsection, we provide the implementation of the $p$-adaptive treecode algorithm, outlined in Algorithm \ref{Al.Treecode}. Firstly, given an initial mesh and a uniform refinement level, a hierarchical tree structure is constructed using the Hierarchy Geometry Tree (HGT) \cite{li2005multi}, which reads the mesh information utilizing the relationships of geometries from different dimensions. 
	
	In HGT, the mesh data is described hierarchically, and $nD$ geometries ($n = 0, 1, 2, 3$) are defined and connected through belonging-to relationships between geometries. Taking a tetrahedron as an example, a point ($0D$ geometry) belongs to an edge of the tetrahedron ($1D$ geometry), which in turn belongs to one of the faces of the tetrahedron ($2D$ geometry), and so on. The advantage of this geometric description lies in its ability to simplify the management of refined or coarsened meshes.
	
	In addition, HGT uses a tree data structure to store and manage geometry information during the mesh refinement or coarsening process. For example, as shown in Fig.~\ref{mesh}, the transition from the left to the middle figure illustrates the uniform refinement of a tetrahedron by connecting the midpoints of its edges, resulting in eight smaller tetrahedrons. Further, the right figure shows the result of locally refining two of these smaller tetrahedrons. The mesh produced after successive refinements can be described using a tree structure, as depicted in Fig.~\ref{octree}, which represents the tree corresponding to the right mesh in Fig.~\ref{mesh}. This non-full octree structure is particularly well-suited for dividing the neighbor and non-neighbor elements of the target element in the treecode algorithm, as illustrated in Fig.~\ref{treecode structure}.
	\begin{figure}[!ht]
		\begin{center}
			\includegraphics[width =0.25\textwidth]{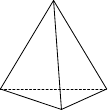}
			\hspace{1em}
			\includegraphics[width =0.25\textwidth]{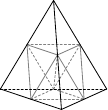}
			\hspace{1em}
			\includegraphics[width =0.25\textwidth]{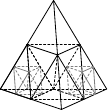}
		\end{center}
		\caption{The parent tetrahedron (left), the eight child tetrahedrons after once refinement (middle), and additional 16 child tetrahedrons after local refinement (right).}
		\label{mesh}
	\end{figure}
	The leaf nodes of the tree exactly cover the entire computational domain, while the near-field (neighbor elements) and far-field (non-neighbor elements) regions in the treecode algorithm also span the full domain. This ensures that the method is perfectly aligned with the tree construction process in the treecode algorithm. Therefore, the HGT-based implementation of the treecode algorithm not only generates the hierarchical tree but also efficiently manages the data, like the set of non-neighbor elements, associated with each element.
	
	\begin{figure}
		\centering
		\includegraphics[width =0.5\textwidth]{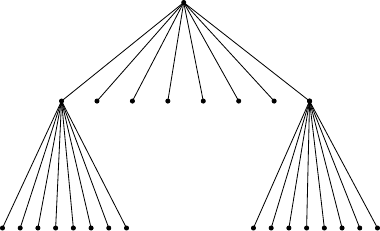}
		\caption{The octree data structure for the mesh in the right of Fig.~\ref{mesh}.}
		\label{octree}
	\end{figure}
	
	In the $p$-adaptive treecode algorithm, we consider all barycenters of leaf elements in HGT as the target points for computing interactions. Before estimating the solution, it is crucial to classify neighbor and non-neighbor elements for each leaf element, as this distinction plays a key role in the efficiency of the treecode algorithm. Selecting a leaf node $\tau_{l,0}$ first and identifying its neighbor elements as those that share common vertex with it. These neighbor elements are treated with direct interactions (near-field interactions) due to their proximity, which satisfies the MAC in Eq.~(\ref{MAC}). Next, move up the tree to the parent node $\tau_{l-1,0}$. The neighbor elements of $\tau_{l-1,0}$, excluding those already marked as neighbors of $\tau_{l,0}$, constitute the non-neighbor elements of $\tau_{l,0}$ at the leaf level, which may be approximated through far-field interactions. Merging upwards again, the non-neighbor elements of $\tau_{l-1,0}$ are also considered as non-neighbor elements of $\tau_{l,0}$ at level $l-1$. This process is repeated recursively up to the root node, allowing us to systematically classify all neighbor and non-neighbor elements for the target element across different levels of the octree. By traversing all leaf nodes, we efficiently collect the necessary near-field and far-field interaction lists for each one.
	
	Furthermore, during the traversal, we would like to gather other element-specific data that does not depend on the target element’s position, such as the quadrature points and their associated weights, as well as the multipole moments $m^\mathbf{k}_c$ of each element. To reduce computational costs, a maximum expansion order, $P_{max}$, is specified, and the corresponding multipole moments are precomputed and stored. These precomputed moments facilitate the efficient approximation of far-field interactions during the evaluation stage.

	Another critical step in the $p$-adaptive algorithm is solution estimation, which requires traversing each target element and dividing the computation into two components: 
	\begin{itemize}
		\item \textbf{Direct summation for neighbor elements:} A sixth-degree Gaussian type quadrature method is utilized to perform direct summation, ensuring accurate near-field calculations, given by
		\begin{equation}
			\int_{\Omega_N} G(\mathbf{x},\mathbf{y})f(\mathbf{y})\dif \mathbf{y} \approx \sum_{K\in\Omega_N}\sum_{l=1}^{24}J_lG(\mathbf{x},\mathbf{y})f(\mathbf{y}_l)|\Omega_K|\omega_l,
			\label{6th-degree gaussian}
		\end{equation}
		where $J_l$ is the jacobian at $\mathbf{y}_l$.
		In addition, the quadrature point in this quadrature rule does not pick the barycenter of the element, thus the singularity in Eq.~(\ref{6th-degree gaussian}) is avoided. 
		\item \textbf{$p$-adaptive multipole expansion for non-neighbor elements:} A multipole expansion of order $p$ is utilized, where $p$ is chosen to meet a specified accuracy criterion (\ref{criterion}).
		
		According to Theorem \ref{Th 1}, we provide a strategy for a non-uniform distribution of the order of multipole expansion order to achieve a given error tolerance $\epsilon$. Those non-neighbor elements $K$ are located at different levels of the tree structure, allowing us to estimate the error introduced by the source elements at each level separately, which yields
		\begin{equation}
			C\left|\sum_{l=1}^M\sum_{K\in\Omega_{F,l}}\dfrac{r_K^{p+1}\left|\Omega_K\right|}{R_K(1-r_K)}  \right|<\epsilon,
		\end{equation}
		where $M$ is the depth of the tree structure, and $\Omega_{F,l}$ is the far-field domain in the $l$th level.
		Therefore, for each level $l$, the accuracy criterion is given by
		\begin{equation}
			\sum_{K\in\Omega_{Fl}}\dfrac{r_K^{p+1}\left|\Omega_K\right|}{R_K(1-r_K)} <\dfrac{\epsilon}{CM},
		\end{equation}
		and the criterion for each source element $K$ can be derived as follows
		\begin{equation}
			\dfrac{r_K^{p+1}\left|\Omega_K\right|}{R_K(1-r_K)} <\dfrac{\epsilon}{nCM},
			\label{criterion}
		\end{equation}
		with $n$ as the number of source elements in the $l$th level. In summary, the fundamental solution $u(\mathbf{x})$ can be evaluated by using a multipole approximation of specified order $p$ when the accuracy criterion (\ref{criterion}) is satisfied.
	\end{itemize}

	\subsection{The flowchart of the algorithm and complexity}
	
	\begin{algorithm}[!ht]
		\caption{Treecode 1: $p$-adaptive treecode algorithm}  
		\begin{algorithmic}[1]
			\State
			\textbf{Input:} The definition of source function: $f(\mathbf{x})$;
			mesh file for $\Omega$;
			the depth of tree: $M$;
			error tolerance: $\epsilon$;
			maximum expansion order: $P_{max}$. \newline
			% the number of tetrahedrons: $N$.
			\textbf{Output:} The distribution of solution $u(\mathbf{x})$.
			\State Construct tree structure based on HGT.
			\For{$i = 1: N$} 
			\State Collect neighbor and non-neighbor elements according to MAC.
			\State Collect quadrature information from leaf layer.
			\State Compute moments of $P_{max}$th-order for each element in the octree.
			\EndFor      
			\For{$i = 1:N$}    
			\State Evaluate the solution $u(\mathbf{x}_i)$ by $p$-adaptive treecode algorithm.
			\For{$K\in\Omega_N$}
			\State $u += u_N(\mathbf{x}_i,\mathbf{y}_K)$ (by direct summation method).
			\EndFor
			\For{$K\in\Omega_F$}
			\State $u +=  u_F(\mathbf{x}_i,\mathbf{y}_K)$ (by multipole expansion with adaptive order).		
			\EndFor
			\EndFor 
			\Function{$u_{F}(\mathbf{x}, \mathbf{y})$}{} 
			
			%\State Compute $r = r_y/R$.
			
			\State Compute $p$ with accuracy criterion (\ref{criterion}).
			\If{$p >P_{max}$}
			\State	$p = P_{max}$.
			\EndIf
			\State Compute the $u_F$ by $p$th-order multipole expansion.
			\EndFunction
		\end{algorithmic}
		\label{Al.Treecode}
	\end{algorithm}

	The flowchart of the $p$-adaptive treecode algorithm is shown in Algorithm \ref{Al.Treecode}, based on the implementation details outlined in Subsection~\ref{sub p-adaptive implementation}. Next, we analyze the computational complexity. The total CPU time of the algorithm is divided into two components: preparation time for information could be pre-computed and the computation time for approximating the solution $u$.
	
	At the preparation stage, the octree is traversed to identify neighbor and non-neighbor elements for $N$ target elements, requiring $O(N\log N)$ operations. Pre-storing the quadrature information for leaf elements has a complexity of $O(N)$, while pre-storing the multipole moments requires $O(p^3N)$ operations. These tasks can be completed during the octree traversal. Therefore, the overall complexity of the preparation phase is $O(N\log N)+O(p^3N)$.
	
	At the computation stage, the time complexity of the multipole expansion for each target element is $O(p^3\log N)$. The direct summation for neighbor elements has a negligible contribution due to the small number of such elements. The use of an adaptive accuracy criterion, where the expansion order $p$ varies across elements, allows $O(p^3\log N)$ to be approximated as $O(\log N)$. Thus, the total time complexity for computing the solution $u$ is $O(NlogN)$.
	
	The overall space complexity can be divided into three components. First, storing the neighbor and non-neighbor elements requires $O(\log N)$ entries per target element, resulting in a total of $O(N \log N)$. Second, storing the quadrature information for leaf elements requires $O(N)$ storage. Third, representing the $p$th-order multipole moments for each element requires $O(p^3)$ storage per element, and with $O(N)$ elements, this contributes $O(p^3 N)$ to the total storage. Therefore, the overall space complexity is $O(N \log N) + O(p^3 N)$.
	
	From the analysis of time and space complexity, it is clear that when using a fixed expansion order $p$, the computational costs of high-order multipole expansion become significant as $p$ increases. In Section \ref{section 4}, we will demonstrate the advantages of the $p$-adaptive treecode algorithm.

	\subsection{A trick on further acceleration with direct summation}
	
	After analyzing the time and space complexity of the algorithm, we determined that pre-storing the multipole expansion for a given maximum order is necessary. However, as the order $p$ increases, both the computational time and storage requirements increase significantly. Moreover, higher-order multipole expansions can sometimes be less efficient than the direct summation method. Therefore, it is essential to set an appropriate maximum expansion order, $P_{max}$.
	To address this issue, we developed an independent program to evaluate the interaction between two elements. The program calculates using multipole expansions of varying orders $p$ and direct summation method. By comparing the CPU time $t_p$ of the $p$th-order multipole expansion with the CPU time $t_{ds}$ of the direct summation method, we determine that direct summation is preferable when $t_p>t_{ds}$, and set $P_{max}=p$.
	
	Algorithm \ref{improved Al} outlines the empirically improved approach for computing the far-field contribution. When the expansion order $p$, determined by the accuracy criterion, is less than $P_{max}$, the $p$th-order multipole expansion is used. Otherwise, if the source element is a leaf node, the direct summation method is applied. If the source element is a non-leaf node, the process is recursively applied by traversing down the octree to its subdomains.
	Since downward traversal may change the number of elements computed by multipole expansion in the process of algorithm execution, we introduce a new accuracy criterion as follows to manage this and ensure the algorithm's accuracy,
	\begin{equation}
		\dfrac{r_K^{p+1}\left|\Omega_K\right|}{R_K(1-r_K)} <\dfrac{\epsilon}{8^{s}nCM},
		\label{new criterion}
	\end{equation}
	where $s$ is the number of downward step.
	\begin{algorithm}[!ht]
		\caption{Treecode 2: further acceleration using direct summation}
		\begin{algorithmic}[1]
			
			\Function{$u_{F}(\mathbf{x}, \mathbf{y})$}{}
			\State Compute $p$ by accuracy criterion (\ref{new criterion}).
			\If{ p $\leq$ $P_{max}$}
			\State Compute the potential by multipole expansion with $p$th-order.
			\ElsIf{$K (\mathbf{y}\in K)$ is the leaf element}
			
			\State Compute the potential by direct summation method.
			\Else
			\For{sub-domain of $K$}
			\State Compute $u_{F}(\mathbf{x}, \mathbf{y})$.
			\EndFor			
			\EndIf
			\EndFunction
		\end{algorithmic}
		\label{improved Al}
	\end{algorithm}
	
	\section{Numerical experiments}
	\label{section 4}
	In this section, we examine the precision and efficiency of the $p$-adaptive treecode algorithm through solving the fundamental solution of Eq.~(\ref{Poisson eq}). 
	We truncate the computational domain from $\mathbb{R}^3$ to a finite domain $\Omega$, and the Poisson equation is given as follows when the solution $u(\mathbf{x})$ demonstrates exponential decay as $|\mathbf{x}|$ becomes sufficiently large.
	\begin{equation}
		\begin{cases}
			-\Delta u(\mathbf{x}) = f(\mathbf{x}), &\text{in } \Omega,\\
			u(\mathbf{x}) = 0 , &\text{on } \partial\Omega.
		\end{cases}	
	\end{equation}
	The corresponding fundamental solution can be written as 
	\begin{equation}
		u(\mathbf{x}) = \int_{\Omega}\dfrac{f(\mathbf{y})}{4\pi|\mathbf{x} - \mathbf{y}|}\dif \mathbf{y}.
		\label{model problem}
	\end{equation}
	
	All numerical experiments are carried out using the C++ package AFEPack \cite{cai2024afepack}. The hardware configuration consists of two AMD EPYC 7713 64-core processors running on the Ubuntu 22.04 operating system. For experiments involving CPU time, no program optimization parameters are applied, and all results are obtained using a single core.
	
	\subsection{Treecode 1: $p$-adaptive treecode algorithm}

	In this section, we aim to evaluate the performance of the proposed $p$-adaptive treecode algorithm in terms of accuracy, computational efficiency, and scalability.
		Firstly, we examine the accuracy of Treecode 1 compared with the uniform order case. The error is given by the $L^2$ error between numerical solution $u_{tc}$ of the treecode algorithm on the mesh with size $h$ and the exact solution $u(\mathbf{x}_i)$, denoted by
	\begin{equation}
		E_1 = \|u_{tc} - u\|_{L^2} = \left(\int_\Omega \left| u_{tc}(\mathbf{x})-u(\mathbf{x})\right|^2\dif \mathbf{x}\right)^{1/2}.
	\end{equation}
	
	In the following experiments, we consider a case defined on truncated domain $\Omega = [-2,2]\times[-2,2]\times[-2,2]$ with uniform tetrahedral mesh. The exact solution of Eq.~(\ref{model problem}) is given by $u(\mathbf{x}) = 2\exp(-\pi (x_1^2 +2x_2^2 + 3x_3^2))$, where $\mathbf{x} = (x_1, x_2, x_3)$. The corresponding source term $f(\mathbf{x})$ is expressed as
	
	\begin{equation}
		f(\mathbf{x}) = -(4\pi^2 x_1^2 + 16\pi^2 x_2^2 + 36\pi^2 x_3^2 - 12\pi)u(\mathbf{x}),
	\end{equation}
	where $f(\mathbf{x})$ attains its maximum value, $F = 24\pi$, at the origin $(0,0,0)$.

	Fig.~\ref{fig:err_convergence} illustrates the error of the uniform order case with $p=2, 10$, and Treecode 1 with a tolerance of $0.1$ and a maximum expansion order of $P_{max}=10$. The figure shows that the mesh refinement has a limited effect on improving the error when $p=2$ in uniform-order algorithm. In particular, when $N$ is large, the error tends to stabilize, making further reduction difficult. The error of the treecode algorithm with a $10$th-order uniform expansion and the $p$-adaptive algorithm decreases as the mesh is refined. This result validates the numerical convergence of the $p$-adaptive treecode and demonstrates that it achieves faster convergence compared to the treecode algorithm using lower-order uniform expansions.
	
	\begin{figure}[!ht]
		\centering
		\includegraphics[width=0.6\textwidth]{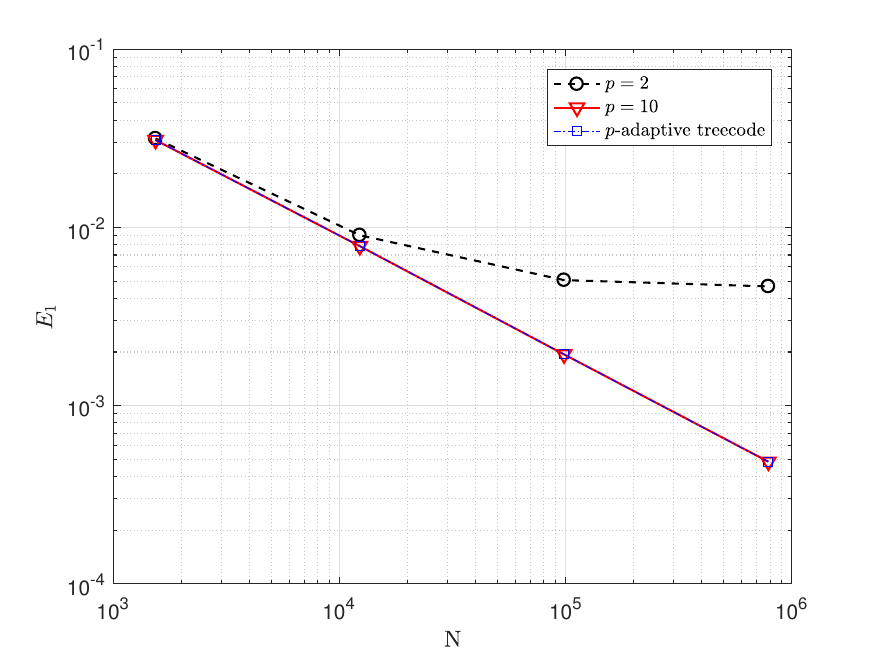}
		\caption{Comparison of $L^{2}$ error scaling with $N$ for treecode algorithm using different order distributions, where $N$ denotes the number of elements in $\Omega$.}
		\label{fig:err_convergence}
	\end{figure}

	Next, we verify the computational efficiency of the $p$-adaptive treecode matches its theoretical complexity of $O(N\log N)$. By comparing its runtime with that of the direct summation method, we demonstrate the significant computational advantage of the $p$-adaptive algorithm for large-scale problems.
	The left figure of Fig.~\ref{time_vs_10} shows the CPU time of Treecode 1 as a function of $N$, in comparison with the direct summation method, which employs a 6th-degree Gaussian numerical quadrature. Due to the high computational complexity of the direct summation method, we computed results only for the first two data points, with subsequent values estimated. From Fig.~\ref{time_vs_10}, it is clear that as $N$ increases, the CPU time for the direct summation method grows rapidly, exhibiting the expected complexity $O(N^2)$. In contrast, the CPU time for the $p$-adaptive treecode grows much more slowly, with a rate slightly faster than $O(N)$. This behavior is consistent with the theoretical complexity $O(N\log N)$ of the treecode algorithm, demonstrating a substantial performance advantage over the direct summation method, particularly for large-scale problems.

	From the results in Fig.~\ref{fig:err_convergence}, it is evident that for element counts ranging from $10^3$ to $10^6$, the uniform expansion method with $p=10$ produces a similar error to the $p$-adaptive algorithm. Therefore, we aims to compare the efficiency of the $p$-adaptive treecode and the uniform-order treecode under the condition of achieving the same accuracy. This comparison highlights the computational savings of the $p$-adaptive approach, particularly in scenarios where achieving higher accuracy is crucial. The right figure of Fig.~\ref{time_vs_10} presents a comparison of the CPU time for the treecode algorithm with uniform and adaptive orders for different numbers of elements $N$. When the number of elements is small $(N<10^4$), the computational efficiency of both methods is comparable, with relatively low computation times, making them suitable for small-scale problems. However, as the number of elements increases, particularly around $N\approx10^6$, the computational cost of the uniform expansion method grows significantly more than that of the $p$-adaptive case. Consequently, for large-scale problems $(N>10^6)$, the proposed treecode algorithm offers a substantial advantage in computational efficiency over the uniform expansion method.
	
	\begin{figure}[!ht]
		\centering
		\includegraphics[width=0.45\textwidth]{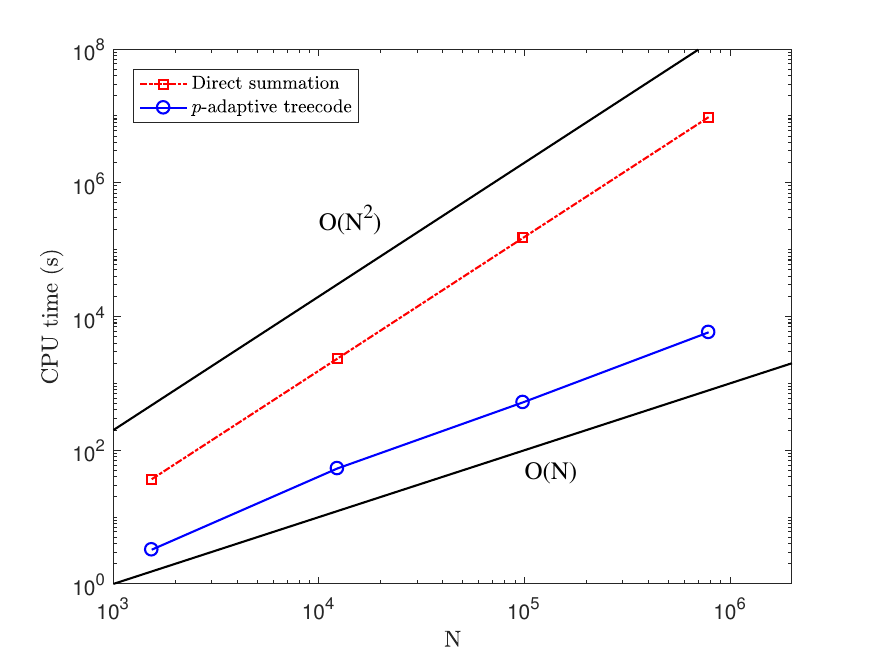}
		\includegraphics[width=0.45\textwidth]{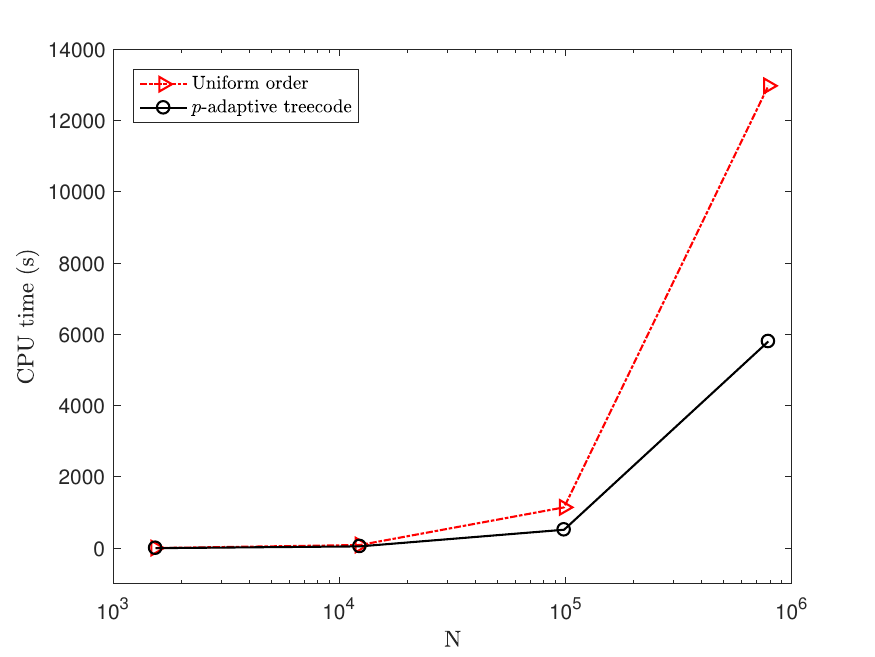}
		\caption{Left: CPU time performance of Treecode 1 with $P_{max} = 10$; right: comparison of CPU time between the Treecode 1 and treecode with uniform expansion order $p=10$ for varying $N$.}
		\label{time_vs_10}
	\end{figure}

	\subsection{Effect of parameters in $p$-adaptive treecode algorithm}
	The following section will examine the impact of parameters in the $p$-adaptive treecode algorithm. To prevent the influence of domain truncation and mesh decomposition, the evaluation of errors will be based on the $L^2$ error between the proposed algorithm Treecode 1 and the direct summation method, which defined by
	\begin{equation}
		E_2 = \|u_{tc}-u_{ds}\|_{L^2}=\left(\int_\Omega \left| u_{tc}(\mathbf{x})-u_{ds}(\mathbf{x})\right|^2\dif \mathbf{x}\right)^{1/2},
	\end{equation}
	where $u_{tc}$ represents the numerical results computed by Treecode 1 algorithm, and $u_{ds}$ is the result computed by $6$th-order Gaussian type quadrature. 
	
	First, we examine the impact of the accuracy parameter $\epsilon$ and the maximum expansion order $P_{max}$ on the $L^2$ error of the $p$-adaptive treecode algorithm. Furthermore, we compare the error convergence behavior of the $p$-adaptive treecode with that of the uniform expansion method as the expansion order increases, providing insights into the relationship between expansion order and accuracy.
	The left figure of Fig.~\ref{fig:t_err} shows the variation in the $L^2$ error of the uniform expansion method as the order $p$ increases, as well as the variation in the $L^2$ error of the Treecode 1 algorithm as the maximum expansion order increases for different $\epsilon$. As $p$ increases, the error $E_2$ of the treecode with uniform expansion decreases consistently, reaching approximately $10^{-11}$ at $p=50$, indicating significant convergence. When different error tolerances $\epsilon$ are applied with $P_{max} = 50$, the error remains stable within the corresponding tolerance range. This demonstrates that the error tolerance $\epsilon$ plays a dominant role in determining the final error when $P_{max}$ is large enough, effectively setting a lower bound for the $L^2$ error. The effect of $P_{max}$ on the error will be explored in the following experiments.
	
	Furthermore, we aim to validate the significant performance advantage of the $p$-adaptive treecode over the uniform expansion method under the condition of achieving the same accuracy. The right figure of Fig.~\ref{fig:t_err} compares the CPU time for the treecode with $p$-adaptive and uniform expansion methods as a function of the error. The error in Treecode 1 is controlled using different $\epsilon$ values, with the maximum expansion order set to $P_{max}=50$. As the error decreases, the CPU time for the uniform expansion method increases sharply, reaching approximately $160$ seconds for an error of $10^{-10}$. In contrast, the $p$-adaptive expansion case shows a much slower increase in CPU time, remaining around $30$ seconds even for the smallest error. This indicates that the $p$-adaptive treecode is significantly more efficient, particularly for high precision. These results demonstrate that while both methods achieve similar errors, the $p$-adaptive treecode provides a substantial performance advantage in terms of computational cost, especially when high accuracy is required.
	
	\begin{figure}[!ht]
		\begin{center}
			\includegraphics[width=0.45\textwidth]{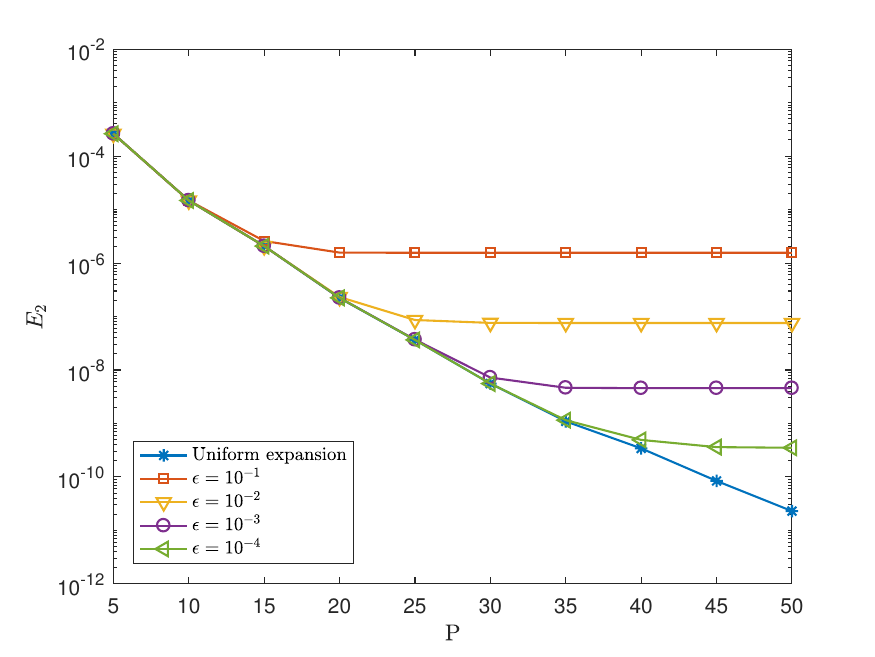}
			\includegraphics[width=0.45\textwidth]{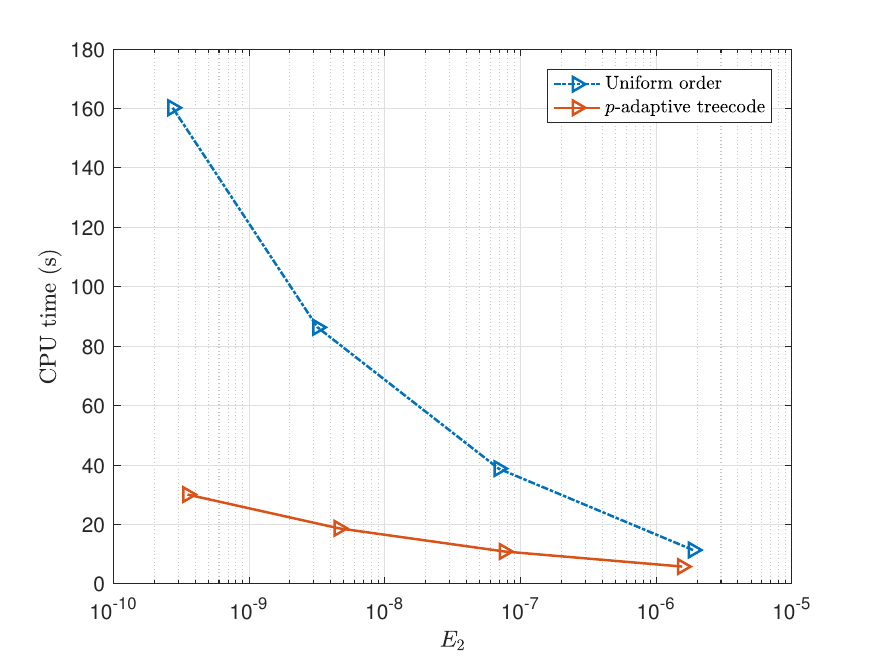}
		\end{center}	
		\caption{Left: the variation in the $L^2$ error of the uniform expansion method with increasing order $p$, as well as the variation in the $L^2$ error of the Treecode 1 algorithm at increasing maximum expansion order for different $\epsilon$; right: efficiency comparison of $p$-adaptive treecode and the uniform order case of various $L^2$ error in a cube with $N = 1536$, $P_{max} = 50$.}
		\label{fig:t_err}
	\end{figure}

	There is something interesting when we examine the impact of varying the maximum expansion order $P_{max}$ in relation to $\epsilon$ on the error for a denser mesh with $N=12,288$. Fig.~\ref{fig: treeocode1_err} shows the $L^2$ error of Treecode 1 as a function of the accuracy parameter $\epsilon$ for different maximum expansion orders $P_{max}$. As the value of $\epsilon$ decreases, the error magnitude in Fig.~\ref{fig: treeocode1_err} also decreases. However, the maximum expansion order limits the influence of $\epsilon$. This is because higher accuracy demands require higher-order expansions. Consequently, when the maximum expansion order cannot meet the accuracy requirements, $\epsilon$ fails to control the error. Furthermore, it is observed that the error decreases consistently with the reduction of $\epsilon$ when $P_{max}=100$. Thus, a maximum expansion order of $P_{max}=100$ is necessary to achieve high accuracy in this case.

	However, the proposed algorithm pre-stores the multipole expansion coefficients of all elements up to the maximum order to enhance computational efficiency. If the maximum order of $100$ is selected, the storage requirements become significant for large-scale problems. Moreover, the overhead of approximate computation using an excessively high expansion order can exceed that of the direct summation method. Therefore, the next stage of the process involves utilizing the enhanced Treecode 2 algorithm to optimize both error performance and computational efficiency.
	
	\begin{figure}[!ht]
		\centering
		\includegraphics[width=0.6\textwidth]{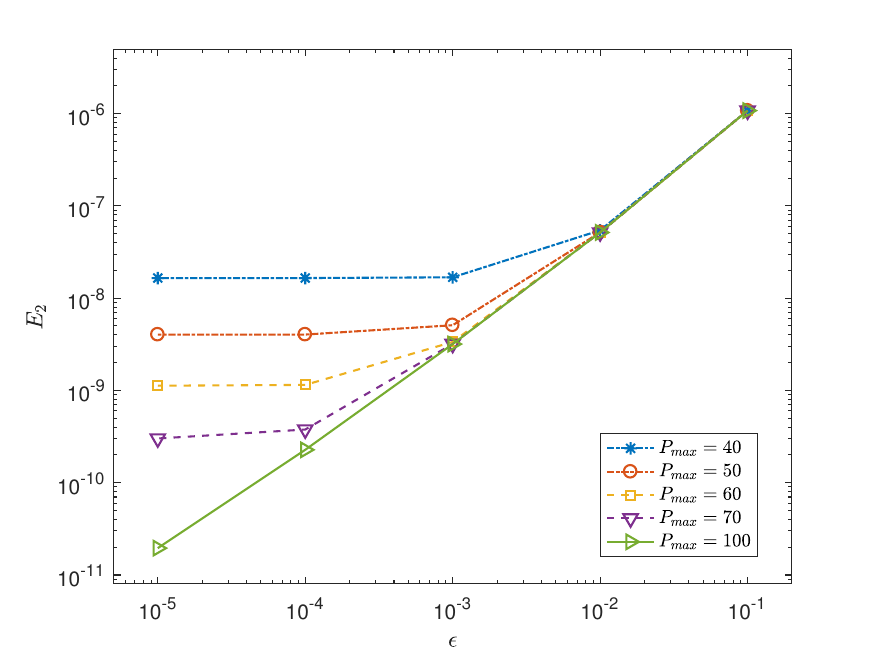}
		\caption{$L^{2}$ error of Treecode 1 with different accuracy parameter $\epsilon $ for different maximum orders $P_{max}$ in the mesh with $N=12288$.}
		\label{fig: treeocode1_err}
	\end{figure}
	
	\subsection{Treecode 2: further acceleration using direct summation}

	In this section, we evaluate the performance of the Treecode 2 algorithm in terms of accuracy and computational cost. In Treecode 2, the computational efficiency of interactions between a single target element and a non-neighbor element was tested using expansions with various orders and direct summation method. This is done by running the standard procedure, which yields a critical expansion order of $25$. It should be noted that this order may vary slightly depending on the experimental setting. Fig.~\ref{fig:Treecode_time_1vs2} illustrates the error and CPU time performance under varying $\epsilon$ values. The left figure shows that the error consistently decreases as $\epsilon$ decreases. The right figure shows that the CPU time increases with decreasing error, but its growth is less significant than in the Treecode 1 algorithm. These results demonstrate that the combination of $p$-adaptive expansion and direct summation enhances the accuracy of treecode algorithm and significantly improves its efficiency.
	
	\begin{figure}[!ht]
		\centering
		\includegraphics[width=0.45\textwidth]{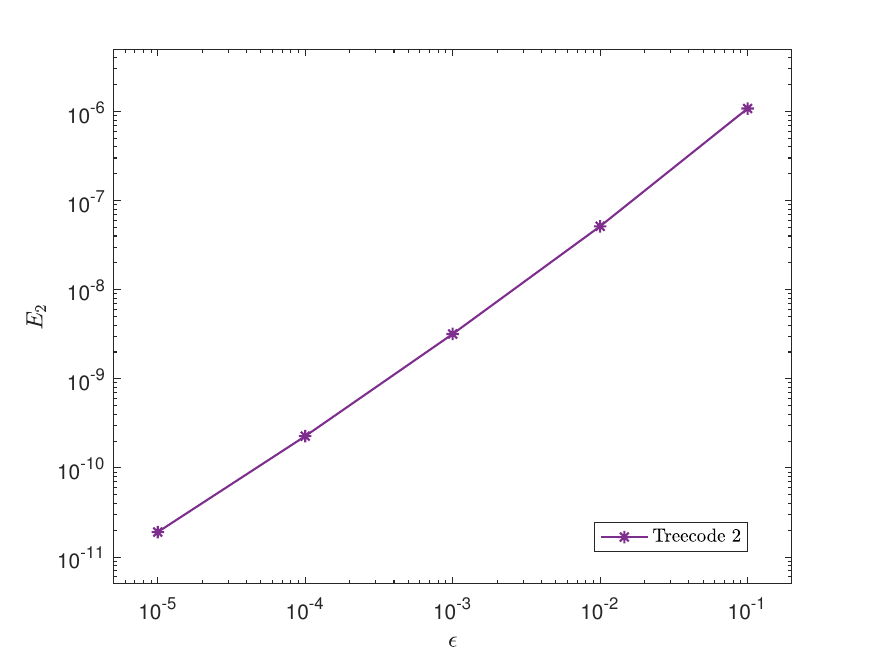} 
		\includegraphics[width=0.45\textwidth]{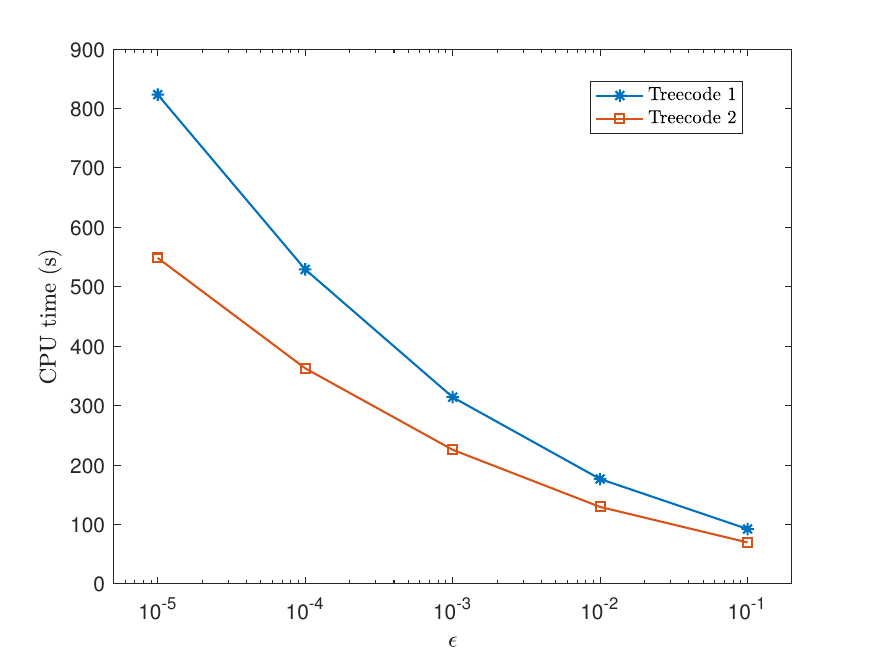}
		\caption{Left: The $L^2$ error of Treecode 2 algorithm varies with parameter $\epsilon$ in the mesh with $N=12288$; right: the CPU time of Treecode 1 and Treecode 2 algorithms varies with accuracy parameter $\epsilon$ in the mesh with $N=12288$.}
		\label{fig:Treecode_time_1vs2}
	\end{figure}
	
	\section{Conclusion}
	\label{section 5}
	
	In this paper, we propose a $p$-adaptive treecode algorithm designed to enhance both computational efficiency and accuracy in the fast evaluation of the fundamental solution to the Poisson equation. Through a systematic error analysis of the treecode algorithm, we develop a strategy for the non-uniform distribution of the multipole expansion order to meet a given error tolerance, thereby improving our understanding of error behavior in the treecode algorithm. Compared to the uniform-order treecode that use uniform expansion orders, our $p$-adaptive treecode significantly enhances efficiency while serving as a robust Poisson solver for applications requiring both high accuracy and computational efficiency. Additionally, our algorithm benefits from an HGT-based implementation and employs a tetrahedral mesh, streamlining the tree construction and domain decomposition processes without imposing restrictions on the shape of the computational region. The efficiency and accuracy of the $p$-adaptive treecode algorithm are validated through various numerical experiments, which consistently demonstrate the desired performance, including a dramatic reduction in computational complexity compared to the uniform expansion method. This makes our algorithm highly competitive for bottleneck problems, such as the demagnetizing field calculation in computational micromagnetics.
	
	A promising direction for future work is to incorporate adaptive mesh refinement, which would dynamically adjust mesh resolution based on local error estimation. This would further improve the algorithm’s efficiency and accuracy, particularly in complex domains, making the proposed treecode more versatile for a broader range of problems.
	
	\section*{Acknowledgement}
	The authors would like to thank Prof. G. Hu from University of Macau for the valuable discussion. The authors would like to thank the support from the Science and Technology Development Fund, Macau SAR (Grant No. 0031/2022/A1).
	\bibliographystyle{plain}
	\bibliography{AdapTreecode}
\end{document}